\documentclass[12pt,a4paper]{amsart}

\usepackage[latin1]{inputenc}
\usepackage{mathptmx}
\usepackage{amssymb}
\usepackage{hyperref}
\usepackage{mathrsfs}
\usepackage[mathscr]{euscript}

\usepackage[ignoreunlbld,norefs,nocites,nomsgs]{refcheck}

\usepackage{pgf, tikz}

\usepackage{fancyvrb} \RecustomVerbatimEnvironment{Verbatim}{Verbatim}
{xleftmargin=15pt, frame=single, fontsize=\small}

\usepackage{pstricks}
\definecolor{light}{gray}{0.9}
\definecolor{medium}{gray}{0.8}

\usepackage{color}

\newtheorem{theorem}{Theorem}

\theoremstyle{definition}
\newtheorem{remark}[theorem]{Remark}

\numberwithin{equation}{section}

\def\ZZ{{\mathbb Z}}

\def\RR{{\mathbb R}}

\def\11{{\mathbb 1}}

\def\cC{{\mathscr C}}
\def\cP{{\mathscr P}}
\def\cU{{\mathscr U}}
\def\cQ{{\mathscr Q}}
\def\cB{{\mathscr B}}
\def\cT{{\mathscr T}}
\def\cK{{\mathscr K}}
\def\cE{{\mathscr E}}
\def\cF{{\mathscr F}}

\def\St{{\textup {St}}}
\def\Sg{{\textup {Sg}}}
\def\SgRev{{\textup {SgRev}}}
\def\relint{{\textup {relint}}}
\def\CW{\textup{CW}}

\def\vol{\operatorname{vol}}

\def\ttt#1{\texttt{#1}}

\let\epsilon=\varepsilon

\allowdisplaybreaks

\begin{document}
\title{Computations of volumes and Ehrhart series in four candidates elections}
\author[W. Bruns]{Winfried Bruns}
\address{Winfried Bruns\\ Universit\"at Osnabr\"uck\\ FB Mathematik/Informatik\\ 49069 Osna\-br\"uck\\ Germany}
\email{wbruns@uos.de}
\author[B. Ichim]{Bogdan Ichim}
\address{Bogdan Ichim\\ Institute of Mathematics ``Simion Stoilow'' of the Romanian Academy\\ C.P. 1-764\\ 010702 Bucharest\\ Romania}
\email{bogdan.ichim@imar.ro}
\author[C. S\"oger]{Christof S\"oger}
\address{Christof S\"oger\\ Universit\"at Osnabr\"uck\\ FB Mathematik/Informatik\\ 49069 Osna\-br\"uck\\ Germany}
\email{csoeger@uos.de}

\subjclass[2010]{52B20, 91B12}

\keywords{rational polytope, volume, Ehrhart series, social choice, Condorcet paradox, Borda paradox}

\maketitle

\begin{abstract}
We describe several experimental results obtained in four candidates
social choice elections. These include the Condorcet and Borda paradoxes, as well as the Condorcet efficiency of plurality voting with runoff. The computations are done by Normaliz. It finds precise probabilities as volumes of polytopes and counting functions encoded as Ehrhart series of polytopes.
\end{abstract}

\section{Introduction}

In \cite[p.\ 382]{LLS} Lepelley, Louichi and Smaoui state:
\begin{quotation}
Consequently, it is not possible to analyze
four candidate elections, where the total number of variables (possible
preference rankings) is 24. We hope that further developments of these algorithms
will enable the overcoming of this difficulty.
\end{quotation}

Normaliz \cite{Nmz} is a software tool that (in particular) may be used for the
computation of volumes and Ehrhart series of rational polytopes. In the
18 years of its existence it has found numerous
applications. For these, as well for its connections to several computer algebra systems see \cite{Nmz}. One of the driving
forces for the improvements of Normaliz in the past years was the desire to
solve the problems raised by Schürmann \cite{Sch}, that is to
compute the volumes and Ehrhart series of certain polytopes related to
social choice.

We believe that the recent development in the algorithms of Normaliz
and its offspring NmzIntegrate have (partially)
solved the problem raised by Lepelley, Louichi and Smaoui.
In \cite[p.\ 388]{BGS} Brandt, Geist and Strobel  write:
\begin{quotation}
To the best of our knowledge, Normaliz is the only program which is able to compute
polytopes corresponding to elections with up to four alternatives.
\end{quotation}
In this paper we present several results of our own computational
experiments with four candidates elections. The background of these experiments, motivated by social choice theory and its connection with the
polytope theory, may be found in Gehrlein and Lepelley \cite{GL},
Lepelley, Louichi and Smaoui, \cite{LLS}, Schürmann \cite{Sch} or Wilson and Pritchard \cite{WP}. For discrete convex geometry we refer the reader to Bruns and Gubeladze \cite{BG}.

As an introductory example we discuss the Condorcet paradox whose
probability of $331/2048\approx 16.2\%$ for four candidates has been known for quite a while
(Gehrlein \cite{G}). We continue with the Condorcet efficiency of
plurality voting (see \cite{Sch}) and complement it by the Condorcet
efficiency of plurality elections with cutoff. While plurality voting has Condorcet efficiency $74.3\%$, a runoff ballot of two candidates increases it to $91.2\%$. The gain is substantial and justifies runoff ballots that are part of many voting procedures. (In this introduction we content ourselves with approximate probabilities; precise rational numbers will be given later on.)

Another problem discussed
in \cite{Sch} is ``plurality versus runoff'', namely the probability
of $75.5\%$ that the plurality winner also wins the runoff.

Next follows a discussion of the four types of antisymmetric relations
between four candidates that can arise from comparisons in majority, and
their probabilities. As we will see, the case (i) of a linear order is
the most likely one by far. The other 3 cases, namely that (ii) there
exists a Condorcet winner, but not a loser, (iii) a loser, but not a
winner, and (iv) neither a winner, nor a loser, have small
probabilities $< 10\%$.

We conclude the list of voting outcomes with the Borda paradoxes. The
strict Borda paradox occurs if the outcome of the pairwise majority
comparison is a linear order and the plurality outcome completely
reverses it. For four candidates its conditional probability is
approximately $0.156\%$. The strong Borda paradox has two variants, namely
(i) that the Condorcet loser wins the plurality, and (ii) that the
Condorcet winner is the last in plurality. As to be expected, their
conditional probabilities are considerably larger, namely about $2.268\%$ and $2.379\%$.
Though one would intuitively expect that these probabilities agree, they
are not equal.

All these events are discussed together with their defining inequalities in
Section \ref{Volumes Computations}. We would like to point out that
Normaliz not only computes the volumes of the polytopes in all cases, but
also their Ehrhart series, and therefore precise numbers of election
results that represent the events under discussion, depending on the number of voters. We  give the
complete numerics only for the Condordet paradox, but all data can be obtained
by request from the authors.

Ehrhart series (see Section \ref{Ehrhart Computations}) are more easily
to compute for closed polytopes than for the  ones that arise if
one excludes ties. However, in many cases the Ehrhart series of the
semiopen polytope can be computed from that of its closure (and
conversely). The crucial condition is that all inequalities, except the
sign conditions,  defining the semiopen polytope are strict, and they are satisfied
with equality by
the election result in which every preference order has the same number
of voters. This follows from a variant of Ehrhart
reciprocity (for example, see \cite[Th. 6.51]{BG}).

Normaliz can do all computations in dimension $24$ (the number of
preference orders for four candidates), but the the computation times
range from seconds (for the Condorcet paradox) to a few days (for the
Condorcet efficiency of the runoff scheme) on a fast machine that allows
$32$ parallel threads. Schürmann \cite{Sch} made the elegant observation
that many computations can be enormously accelerated if one uses the
symmetry that is inherent in many polytopes. Only two of the polytopes
discussed in this paper do not allow this approach, namely linear order
and (consequently) the strict Borda paradox. Symmetrization, which
requires the computation of weighted Ehrhart series, shrinks computation
times from days to hours, minutes, or even tenths of seconds, depending
on the example. Symmetrization is discussed in Section \ref{Exploit}, and
all relevant computation data are listed in Section \ref{CompRep}.

The reader who is interested in a deeper understanding of the
mathematics and algorithms of Normaliz and NmzIntegrate is refereed to
the articles by Bruns with Koch \cite{BK}, Ichim \cite{BI}, Hemmecke, Ichim, Köppe and Söger\cite{BHIKS}, Ichim and Söger \cite{BIS}, Söger \cite{BS} and Sieg and Söger \cite{BSS}.

\section{Polytopes in four candidates elections and their volumes}\label{Volumes Computations}

\subsection{Voting schemes and volumes of rational polytopes}
We briefly sketch the connection between rational polytopes and
social choice, referring the reader to \cite{GL},
\cite{LLS}, \cite{Sch} or \cite{WP} for details and a more extensive
treatment. As an introductory example we first consider the well-known Condorcet paradox.
The polytopes $\cP$ associated to it and to other voting events are \emph{semiopen}: $\cP$ is the bounded set of solutions of a system of linear equations and inequalities in which some of the inequalities may be strict.

Consider an election in which each of the $k$ voters fixes a
linear preference order of $n$ candidates. In other words,
voter $i$ chooses a linear order $j_1\succ_i\dots\succ_i j_n$
of the candidates $1,\dots,n$. There are $N=n!$ such linear orders
that we list lexicographically. The
election result is the  $N$-tuple $(v_1,\dots,v_N)$ in which
$v_p$ is the number of voters that have chosen the preference
order $p$. Then $v_1+\dots+v_N=k$, and $(v_1,\dots,v_N)$ can be
considered as a lattice point in the positive orthant
$\RR_+^N$ of $\RR^B$, more precisely, as a lattice point in the \emph{simplex}
$$
\cU_k^{(n)}=\RR_+^N\cap A_k=k\bigl(\RR_+^N\cap A_1\bigr)=k\cU^{(n)}
$$
where $A_k$ is the hyperplane defined by $v_1+\dots+v_N=k$, and
$\cU^{(n)}=\cU_1^{(n)}$ is the unit simplex of dimension $N-1$
naturally embedded in $N$-space. ($\cU^{(n)}$ is the convex hull of the unit vectors; see Figure \ref{fig_init}.)
\begin{figure}[hbt]
\begin{center}
\tikzset{facet style/.style={opacity=1.0,very thick,line,join=round}}
\begin{tikzpicture}[x  = {(-0.5cm,-0.5cm)},
y  = {(0.9659cm,-0.25882cm)},
z  = {(0cm,1cm)},
scale = 2]
\draw [->] (-0.5, 0, 0) -- (1.7,0,0) node at (1.9,0,0) {$x_1$} node at (1, -0.1,0){$1$};
\draw [->] (0, -0.5, 0) -- (0,1.5,0) node at (0,1.7,0) {$x_2$} node at (-0.1, 1,0.1){$1$};
\draw [->] (0, 0, -0.5) -- (0,0,1.5) node at (0,0,1.6){$x_3$}node at (-0.1, -0.15,1){$1$};
\draw[thick] (1,0,0) -- (0,1,0) -- (0,0,1) -- cycle;
\end{tikzpicture}
\caption{The unit simplex in $\RR^3$}\label{fig_init}
\end{center}
\end{figure}
The further  discussion is
based on the \emph{Impartial Anonymous Culture} assumption that
all lattice points in the simplex $\cU_k^{(n)}$ have equal
probability of being the outcome of the election.

We fix a specific outcome $v=(v_1,\dots,v_N)$. By the \emph{majority rule}
candidate $j$ beats candidate $j'$ in the election with result $v$
if
\begin{equation}
\#\{i:j \succ_i j':i=1,\dots,k\} > \#\{i:j' \succ_i j:i=1,\dots,k\}.\label{cond1}
\end{equation}
As the Marquis de Condorcet \cite{C} observed, the relation ``beats'' is
nontransitive in general, and one must ask for the probability
of Condorcet's paradoxon, namely an outcome  without a
Condorcet winner, where candidate $j$ is a \emph{Condorcet
winner}  if $j$ beats all other candidates $j'$. The Condorcet winner will sometimes be denoted $\CW$ in the following. Given the number $k$ of voters, let
$p_{\CW}^{(n,j)}(k)$ denote the probability that candidate $j$ is
the Condorcet winner, and $p_{\CW}^{(n)}(k)$ the probability that
there is a Condorcet winner at all. By symmetry and by mutual
exclusion, $p_{\CW}^{(n)}(k)=n\cdot p_{\CW}^{(n,j)}(k)$.

Now, if we assume that the number $k$ of voters is very large, then we are mainly interested in
the limit
$$
p_{\CW}^{(n)}=\lim_{k\to\infty} p_{\CW}^{(n)}(k)= n\lim_{k\to\infty} p_{\CW}^{(n,j)}(k)=n\cdot  p_{\CW}^{(n,j)}.
$$

Let us fix candidate $1$. It is not hard to see that the $n-1$
inequalities \eqref{cond1} for $j=1$ and $j'=2,\dots,n$
constitute homogeneous linear inequalities in the variables
$v_1,\dots,v_N$. Together with the sign conditions $v_i\ge 0$ they
define a semi-open subpolytope $\cC^{(n)}_k$ of $\cU_k^{(n)}$.
Then
\begin{equation}
p_{\CW}^{(n,1)}=\lim_{k\to\infty}\frac{\#(\cC_k^{(n)}\cap\ZZ^N)}{\#(\cU_k^{(n)}\cap \ZZ^N)}=
\frac{\vol\cC_1^{(n)}}{\vol\cU_1^{(n)}}=\vol\overline {\cC}^{(n)}\label{Cond2}
\end{equation}
where $\overline{\phantom{C}}$ denotes closure and
$\cC^{(n)}=\cC_1^{(n)}$. For the validity of \eqref{Cond2} note
that we work with the lattice normalized volume in which the
unit simplex has volume $1$.

In the case of two candidates, Concordet's paradox cannot occur
(if one excludes draws), and for three candidates the relevant
volume is not hard to compute, even without a computer (see Gehrlein and Fishburn \cite{GF}).

The situation changes significantly for four candidates since $\cC^{(4)}$ has
dimension $23$ and $234$ vertices. \emph{From now on we will simplify our notation and omit often the superscript $(4)$ when dealing with four candidates.}
As a subpolytope of $\cU$, $\cC$ is cut out by the inequalities
$\lambda_i(v)>0$, $i=1,2,3$ whose coefficients are in the first
$3$ rows displayed in Table \ref{ineq1}. For the assignment of
indices the preference orders are listed lexicographically,
starting with $1\succ 2 \succ 3\succ 4$ and ending with $4\succ
3 \succ 2\succ 1$.
\begin{table}[hbt]
{\small \tabcolsep=1.3pt
\begin{tabular}{rrrrrrrrrrrrrrrrrrrrrrrrr}
$\lambda_1$:&\ \ \ 1&\ \ \ 1&\ \ \ 1&\ \ \ 1&\ \ \ 1&\ \ \ 1&$-1$&$-1$&$-1$&$-1$&$-1$&$-1$&1&1&$-1$&$-1$&1&$-1$&1&1&$-1$&$-1$&1&$-1$\\
$\lambda_2$:&1&1&1&1&1&1&1&1&$-1$&$-1$&1&$-1$&$-1$&$-1$&$-1$&$-1$&$-1$&$-1$&1&1&1&$-1$&$-1$&$-1$\\
$\lambda_3$:&\rule[-1.5ex]{0ex}{2ex}1&1&1&1&1&1&1&1&1&$-1$&$-1$&$-1$&1&1&1&$-1$&$-1$&$-1$&$-1$&$-1$&$-1$&$-1$&$-1$&$-1$\\ \hline
$\lambda_4$:&\rule{0ex}{3ex}$1$&$1$&$1$&$1$&$1$&$1$&-1&-1&-1&-1&-1&-1&0&0&0&0&0&0&0&0&0&0&0&0\\
$\lambda_5$:&1&1&1&1&1&1&0&0&0&0&0&0&$-1$&$-1$&$-1$&$-1$&$-1$&$-1$&0&0&0&0&0&0\\
$\lambda_6$:&1&1&1&1&1&1&0&0&0&0&0&0&0&0&0&0&0&0&$-1$&$-1$&$-1$&$-1$&$-1$&$-1$
\end{tabular}}
\vspace*{2ex} \caption{Inequalities for $\cC$ and
$\cE$}\label{ineq1}
\end{table}

Normaliz computes
$$
\vol\cC=\frac{1717}{8192}
$$
in a few seconds. (The combinatorial data of all polytopes discussed and the computation times are listed in Section \ref{CompRep}.) It follows that $p_\CW=1717/2048\approx
0.8384$. This value was first determined by Gehrlein in \cite{G}.

\subsection{Plurality rule and runoff}
The simplest way out of the dilemma that there may not exist a
Condorcet winner is \emph{plurality voting}: candidate $j$ is
the \emph{plurality winner} if $j$ has more first places in the
preference orders of the voters than any of the other $n-1$
candidates. A problem discussed in \cite{Sch} is \emph{plurality
voting versus plurality runoff}. It goes as follows. In the
first round of the election the two top candidates in plurality
are selected, and in the second round the preference
orders are restricted to these two candidates. In order to
model this situation by inequalities one must fix an outcome of
the first round of plurality voting,
for example we may assume that candidate $1$ is the winner of the first round and candidate $2$ is placed second after the first round.
The chosen outcome gives rise to $n-1$
inequalities. Then the $n$-th inequality expresses that $1$ is also
the winner of the second round. The volume of the corresponding
polytope gives the probability of this event. By mutual
exclusion and symmetry, we must multiply the volume by $n(n-1)$ in
order to obtain the probability for the event that the winner
of the first plurality round wins after runoff.

As a subpolytope of $\cU$, the polytope $\cQ$ is
defined by the inequalities in Table \ref{ineq2}.
\begin{table}[hbt]
{\small \tabcolsep=1.3pt
\begin{tabular}{rrrrrrrrrrrrrrrrrrrrrrrr}
1&1&1&1&1&1&$-1$&$-1$&$-1$&$-1$&$-1$&$-1$&0&0&0&0&0&0&0&0&0&0&0&0\\
0&0&0&0&0&0&1&1&1&1&1&1&$-1$&$-1$&$-1$&$-1$&$-1$&$-1$&0&0&0&0&0&0\\
0&0&0&0&0&0&1&1&1&1&1&1&0&0&0&0&0&0&$-1$&$-1$&$-1$&$-1$&$-1$&$-1$\\
\ \ \ 1&\ \ \ 1&\ \ \ 1&\ \ \ 1&\ \ \ 1&\ \ \ 1&$-1$&$-1$&$-1$&$-1$&$-1$&$-1$&1&1&$-1$&$-1$&1&$-1$&1&1&$-1$&$-1$&1&$-1$
\end{tabular}}
\vspace*{2ex} \caption{Inequalities for
$\cQ$}\label{ineq2}
\end{table}
The volume is
$$
\vol \cQ=\frac{9185069468583833}{146081389744226304}.
$$
The total probability of the event that the winner
of the first plurality round wins after runoff is
$$
12\cdot \vol \cQ=\frac{9185069468583833}{12173449145352192} \approx
0.7545.$$
Therefore, the probability of the failure of the winner of the first
round to win the runoff is
$$
1-\frac{9185069468583833}{12173449145352192} \approx
0.2455,
$$
in accordance with the results of \cite{Sch} for a similar
model. This computation was first performed by De Loera, Dutra,
K\"{o}ppe, Moreinis, Pinto and Wu in \cite{Latte}, where LattE
Integrale \cite{LatInt} was used for the volume computation.

\subsection{Condorcet efficiency} The last problem discussed in \cite{Sch} is the  \emph{Condorcet efficiency} of plurality voting. It is the conditional
probability that the Condorcet winner, provided that such
exists, is elected by plurality voting, as $k\to\infty$. (Similarly one defines the Condorcet efficiency of other voting schemes.) Therefore one must compute the
probability of the event that candidate $j$ is both the Condorcet
winner and the plurality winner. By
symmetry, one can assume $j=1$. The semi-open
polytope $\cE_k^{(n)}$, whose lattice points represent this
expected outcome, is cut out from $\cC_k^{(n)}$ by $n-1$
further inequalities saying that $1$ has more first places than any of
the other $n-1$ candidates. Thus one obtains
$$
\frac{n\vol \cE^{(n)}}{p_{\CW}^{(n)}}
$$
as the Condorcet efficiency of plurality voting where
$\cE^{(n)}=\cE_1^{(n)}$.

The extra $3$ inequalities $\lambda_i(v)>0$, $i=4,5,6$, given
in the last $3$ lines of Table \ref{ineq1} increase the
complexity of the polytope $\cE$ enormously in comparison to $\cC$. Nevertheless, Normaliz
computes the volume in moderate time.
We have obtained
$$
\vol \cE=\frac{10658098255011916449318509}{68475651442606080000000000},
$$
so that the Condorcet efficiency of plurality voting turns out
to be
$$
\frac{4\vol \cE}{p_\CW}=
\frac{10658098255011916449318509}{14352135440302080000000000}
\approx 0.7426,
$$
in perfect accordance with \cite{Sch}.

\begin{remark} Schürmann \cite{Sch} used variants of the polytopes $\cQ$ and $\cE$. Our choices (which are demanding slightly more computational resources) avoid inclusion--exclusion calculations that would come up in Section \ref{Ehrhart Computations}.
\end{remark}

It is interesting to compare the Condorcet efficiency of plurality voting to the  \emph{the Condorcet efficiency of the runoff
voting scheme}. In other words, given that there exists a Condorcet winner, what is the
probability that she or he is at least second in plurality?

As above, let us assume that candidate $1$ is the Condorcet winner.  Then there are $n$ possible cases.
The first case is when candidate $1$ is the plurality winner as well, and this case was studied above.
The other $n-1$ cases are associated to the event  that  candidate $j$ wins or ties candidate $1$ in plurality (where $j\neq 1$),
while candidate $1$ wins the plurality voting against all other candidates.  By symmetry, these $n-1$ cases are identical and yield the same volume.
The $n$ cases are mutually disjoint and exhaust all the possibilities. (Disjointness is important in Section \ref{Ehrhart Computations} for avoiding complicated inclusion--exclusion calculations.) The semi-open
polytope $\cF_{1,j}^{(n)}$, whose lattice points represent the outcome of the cases $j=2,\ldots,n$, is cut out from $\cC_k^{(n)}$ by the  closed condition that candidate $j$ wins against or ties with candidate $1$ in plurality and by  $n-2$
inequalities saying that $1$ has more first places than the remaining $n-2$ candidates. Thus one obtains
$$
\frac{n\vol \cE^{(n)}+n(n-1)\vol \cF^{(n)}}{p_{\CW}^{(n)}}
$$
as the the Condorcet efficiency of the runoff, where
$\cF^{(n)}=\cF_{1,2}^{(n)}$.

As a subpolytope of $\cU$, the polytope $\cF$ is
defined by the inequalities in Table \ref{ineqadded}, where the last line should be interpreted as $\overline{-\lambda_4}(v)\ge 0$, expressing the condition that candidate 1 does not beat candidate 2 in plurality.
\begin{table}[hbt]
{\small \tabcolsep=1.3pt
\begin{tabular}{rrrrrrrrrrrrrrrrrrrrrrrrr}
$\lambda_1$:&\ \ \ 1&\ \ \ 1&\ \ \ 1&\ \ \ 1&\ \ \ 1&\ \ \ 1&$-1$&$-1$&$-1$&$-1$&$-1$&$-1$&1&1&$-1$&$-1$&1&$-1$&1&1&$-1$&$-1$&1&$-1$\\
$\lambda_2$:&1&1&1&1&1&1&1&1&$-1$&$-1$&1&$-1$&$-1$&$-1$&$-1$&$-1$&$-1$&$-1$&1&1&1&$-1$&$-1$&$-1$\\
$\lambda_3$:&\rule[-1.5ex]{0ex}{2ex}1&1&1&1&1&1&1&1&1&$-1$&$-1$&$-1$&1&1&1&$-1$&$-1$&$-1$&$-1$&$-1$&$-1$&$-1$&$-1$&$-1$\\ \hline
$\lambda_5$:&1&1&1&1&1&1&0&0&0&0&0&0&$-1$&$-1$&$-1$&$-1$&$-1$&$-1$&0&0&0&0&0&0\\
$\lambda_6$:&1&1&1&1&1&1&0&0&0&0&0&0&0&0&0&0&0&0&$-1$&$-1$&$-1$&$-1$&$-1$&$-1$\\ \hline
$\overline{-\lambda_4}$:&\rule{0ex}{3ex}$-1$&$-1$&$-1$&$-1$&$-1$&$-1$&1&1&1&1&1&1&0&0&0&0&0&0&0&0&0&0&0&0\\
\end{tabular}}
\vspace*{2ex} \caption{Inequalities for
$\cF$}\label{ineqadded}
\end{table}

The volume is
$$
\vol \cF=\frac{7280153240719060220104571}{616280862983454720000000000}.
$$
Finally, the Condorcet efficiency of the runoff is
$$
\frac{4\vol \cE+12\vol \cF }{p_\CW}=
\frac{19627224002877404784030049}{21528203160453120000000000}
\approx 0.9117.
$$

\subsection{Condorcet classification}\label{subsection:Cond_class} In this subsection we classify the asymmetric relations between four candidates given by the majority rule. To the best of our knowledge the computations of the probabilities of the classes is new.

Let us first outline a duality argument that will be used several times below. Consider an election result $v=(v_1,\dots,v_N)$ for $n$ candidates, where $N=n!$, the $N$ preference orders $\pi_1,\dots,\pi_N$ are listed in some order, and $v_i$ is the number of voters of $\pi_i$. Each preference order has an inverse $c(\pi)$ that ranks the candidates in inverse order relative to $\pi$: the inverse order to $1 \succ 2 \succ 3\succ 4$ is $4\succ 3 \succ 2\succ 1$ etc. The assignment $\pi\to c(\pi)$ defines a permutation of the sequence $1,\dots,N$, sending $i$ to the index of $c(\pi_i)$. The induced permutation of the coordinates of $\RR^N$ is called \emph{inversion of preference orders}. It inverts all comparisons in majority. In particular it turns a Condorcet winner into a Condorcet loser, and conversely.

The results of the $n$ candidates elections may  be classified in two main categories:
\begin{enumerate}
\item[(A)] There exists a Condorcet winner. As seen above, in the case of four candidates elections the results fall into this category  with probability
 $$P(\exists \text{ Condorcet Winner})=1717/2048\approx 0.8384.$$
\item[(B)] There exists no Condorcet winner. For four candidates elections the results fall into this category  with probability
$$P(\nexists \text{ Condorcet Winner})=1-1717/2048=331/2048\approx 0.1616.$$
\end{enumerate}

We refer the reader to Gehrlein and Lepelley \cite[Section 3.2.1]{GL} or \cite{GL2} for a discussion of the three candidates situation. Note that in the three candidates scenario the event that there exist a linear order on the result of the majority voting is the same as the event that there exists a Condorcet winner, since the other two candidates are automatically ordered. This no longer true for four or more candidates. Even if a Condorcet winner exists, the  remaining candidates need not to be linearly ordered. Therefore, we need to further refine our classification.

We discuss the case of four candidates in detail:
\begin{enumerate}
	
\item[(A)] Assume that a Condorcet winner (CW) does exist.  This situation must be split into two subcategories:
\begin{enumerate}
\item[(1)] The result of the majority voting defines a linear order of the candidates. This further implies (independently of the number of candidates $n$) that there also exists a Condorcet loser (CL).

\item[(2)] There exist no linear order on the result of the majority voting. In the (particular) case of four candidates elections this is equivalent to saying that there exists a cycle of length three among
the lower candidates (i.e., the three candidates Condorcet Paradox) or that there exists no Condorcet loser.
\end{enumerate}

\item[(B)] Now assume the case that a Condorcet winner does not exist.  This situation must also  be split into two subcategories:
\begin{enumerate}
\item[(1)] There exists a cycle of order three among the candidates and a Condorcet loser. Inversion of preference  orders turns this case into (A2). In particular they have the same probability.

\item[(2)] There exist a cycle of length four among the candidates or (equivalently) there exists no Condorcet loser. This condition defines only $4$ of the $6$ relations between the candidates, but it easy to check that all 4 possibilities for the remaining $2$ relations are equivalent up to a renaming of the candidates.
\end{enumerate}
\end{enumerate}

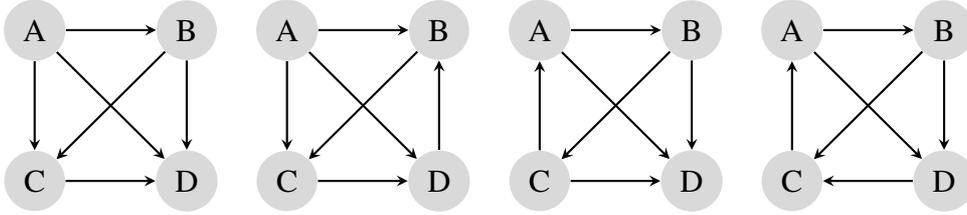
\begin{figure}[hbt]
\begin{center}
\begin{tikzpicture}
\tikzstyle{every node} = [circle, fill=gray!30]
\node (a) at (0,2) {A};
\node (c) at (0,0) {C};
\node (b) at (2,2) {B};
\node (d) at (2,0) {D};
\foreach \from/\to in {a/b, a/c, a/d, b/c, b/d, c/d}
\draw [->, >=stealth, thick] (\from) -- (\to);	
\end{tikzpicture}
\hspace{3mm}	
\begin{tikzpicture}
\tikzstyle{every node} = [circle, fill=gray!30]
\node (a) at (0,2) {A};
\node (c) at (0,0) {C};
\node (b) at (2,2) {B};
\node (d) at (2,0) {D};
\foreach \from/\to in {a/b, a/c, a/d, b/c, d/b, c/d}
\draw [->, >=stealth, thick] (\from) -- (\to);	
\end{tikzpicture}
\hspace{3mm}
\begin{tikzpicture}
\tikzstyle{every node} = [circle, fill=gray!30]
\node (a) at (0,2) {A};
\node (c) at (0,0) {C};
\node (b) at (2,2) {B};
\node (d) at (2,0) {D};
\foreach \from/\to in {a/b, c/a, a/d, b/c, b/d, c/d}
\draw [->, >=stealth, thick] (\from) -- (\to);	
\end{tikzpicture}
\hspace{3mm}
\begin{tikzpicture}
\tikzstyle{every node} = [circle, fill=gray!30]
\node (a) at (0,2) {A};
\node (c) at (0,0) {C};
\node (b) at (2,2) {B};
\node (d) at (2,0) {D};
\foreach \from/\to in {a/b, a/d, b/c, b/d, d/c, c/a}
\draw [->, >=stealth, thick] (\from) -- (\to);	
\end{tikzpicture}	
\end{center}
\caption{Oriented graphs representing the Condorcet classes of four candidates with respect to the relation given by the majority rule}\label{graphs}
\end{figure}

In order to compute the probabilities of the $4$ classes, we consider the polytope $\cT$ which corresponds to the event that candidate $1$ beats candidates $2,3,4$, candidate $2$ beats candidates $3,4$ and candidate $3$ beats candidates $4$. In other words, $\cT$ represents the linear order.
As a subpolytope of $\cU$, the polytope $\cT$ is
defined by the inequalities in Table \ref{linear_order}. (Note that $\beta_i=\lambda_i$ for $i=1,2,3$.)
\begin{table}[hbt]
{\small \tabcolsep=1.3pt
\begin{tabular}{rrrrrrrrrrrrrrrrrrrrrrrrr}
$\beta_1$:&$ 1$&$ 1$&$ 1$&$ 1$&$ 1$&$ 1$     &$-1$&$-1$&$-1$&$-1$&$-1$&$-1$      &$ 1$&$ 1$&$-1$&$-1$&$ 1$&$-1$        &$ 1$&$ 1$&$-1$&$-1$&$ 1$&$-1$\\
$\beta_2$:&$ 1$&$ 1$&$ 1$&$ 1$&$ 1$&$ 1$     &$ 1$&$ 1$&$-1$&$-1$&$ 1$&$-1$      &$-1$&$-1$&$-1$&$-1$&$-1$&$-1$        &$ 1$&$ 1$&$ 1$&$-1$&$-1$&$-1$\\
$\beta_3$:&$ 1$&$ 1$&$ 1$&$ 1$&$ 1$&$ 1$     &$ 1$&$ 1$&$ 1$&$-1$&$-1$&$-1$      &$ 1$&$ 1$&$ 1$&$-1$&$-1$&$-1$        &$-1$&$-1$&$-1$&$-1$&$-1$&$-1$\\
$\beta_4$:&$ 1$&$ 1$&$-1$&$-1$&$ 1$&$-1$     &$ 1$&$ 1$&$ 1$&$ 1$&$ 1$&$ 1$      &$-1$&$-1$&$-1$&$-1$&$-1$&$-1$        &$ 1$&$-1$&$ 1$&$ 1$&$-1$&$-1$\\
$\beta_5$:&$ 1$&$ 1$&$ 1$&$-1$&$-1$&$-1$     &$ 1$&$ 1$&$ 1$&$ 1$&$ 1$&$ 1$      &$ 1$&$-1$&$ 1$&$ 1$&$-1$&$-1$        &$-1$&$-1$&$-1$&$-1$&$-1$&$-1$\\
$\beta_6$:&$ 1$&$-1$&$ 1$&$ 1$&$-1$&$-1$     &$ 1$&$-1$&$ 1$&$ 1$&$-1$&$-1$      &$1$&$1$&$1$&$1$&$1$&$1$        &$-1$&$-1$&$-1$&$-1$&$-1$&$-1$
\end{tabular}}
\vspace*{2ex} \caption{Inequalities for
$\cT$}\label{linear_order}
\end{table}

We have obtained
$$
\vol \cT=\frac{5507086513}{173946175488},
$$
Since a set of $4$ elements admits $24$ possible linear orders, the probability to have a linear order on the result of the majority voting is
$$
P(\exists \text{ CW},\exists \text{ CL})=\frac{5507086513}{7247757312}\approx 0.7598.
$$
It follows that the probability that a Condorcet winner does exists, and still there exist no linear order on the result of the majority voting is
$$
P(\exists \text{ CW},\nexists \text{ CL})=\frac{ 1717}{2048}-\frac{5507086513}{7247757312}=\frac{569280335}{7247757312}\approx 0.07855.
$$
By the duality argument observed in the case (B1)  the probability that a Condorcet loser does exists, but no Condorcet winner, is
$$
P(\nexists \text{ CW},\exists \text{ CL})=\frac{569280335}{7247757312}\approx 0.07855
$$
as well. The probability of the remaining class (B2), the existence of a $4$-cycle, is
$$
P(\nexists \text{ CW},\nexists \text{ CL})=1-\frac{5507086513}{7247757312}-2*\frac{569280335}{7247757312}=\frac{602110129}{7247757312}\approx 0.0831.
$$

As a test for the correctness of the algorithm, we have nevertheless computed the probability of a $4$-cycle directly. To this end, we consider the polytope $\cK$ corresponding to the event that candidate $1$ beats candidates $2$, candidate $2$ beats candidates $3$, candidate $3$ beats candidates $4$ and candidate $4$ beats candidates $1$.
As a subpolytope of $\cU$, the polytope $\cK$ is
defined by the inequalities in Table \ref{cycle}.
\begin{table}[hbt]
{\small \tabcolsep=1.3pt
\begin{tabular}{rrrrrrrrrrrrrrrrrrrrrrrrr}
$\beta_1$:&$ 1$&$ 1$&$ 1$&$ 1$&$ 1$&$ 1$     &$-1$&$-1$&$-1$&$-1$&$-1$&$-1$      &$ 1$&$ 1$&$-1$&$-1$&$ 1$&$-1$        &$ 1$&$ 1$&$-1$&$-1$&$ 1$&$-1$\\
$\beta_4$:&$ 1$&$ 1$&$-1$&$-1$&$ 1$&$-1$     &$ 1$&$ 1$&$ 1$&$ 1$&$ 1$&$ 1$      &$-1$&$-1$&$-1$&$-1$&$-1$&$-1$        &$ 1$&$-1$&$ 1$&$ 1$&$-1$&$-1$\\
$\beta_6$:&$ 1$&$-1$&$ 1$&$ 1$&$-1$&$-1$     &$ 1$&$-1$&$ 1$&$ 1$&$-1$&$-1$      &$1$&$1$&$1$&$1$&$1$&$1$        &$-1$&$-1$&$-1$&$-1$&$-1$&$-1$\\
$\kappa$:&$-1$&$-1$&$-1$&$-1$&$-1$&$-1$     &$-1$&$-1$&$-1$&$ 1$&$ 1$&$ 1$      &$-1$&$-1$&$-1$&$ 1$&$ 1$&$ 1$        &$ 1$&$ 1$&$ 1$&$ 1$&$ 1$&$ 1$
\end{tabular}}
\vspace*{2ex} \caption{Inequalities for
$\cK$}\label{cycle}
\end{table}
We have obtained
$$
\vol \cK=\frac{602110129}{43486543872},
$$
Since a set of $4$ elements admits $6$ possible cycles of length four among the elements (if we fix one element there are 6 possible linear orders among the remaining three elements), one obtains exactly the same probability for the class (B2) that was computed indirectly above.

\subsection{Borda paradoxes} In this subsection we study a family of voting paradoxes, first introduced by the Chevalier de Borda in \cite{B}.
To the best of our knowledge all volume computations in this subsection are new.

The \emph{strict Borda paradox} appears in a voting situation when there is a complete reversal of the ranking of candidates  given by the \emph{majority voting} and \emph{plurality voting}.
In order to model this situation by inequalities one must fist assume that there exist a \emph{linear order} on the result of the majority voting that involves all $n$ candidates, say
$1,\dots,n$ in this order (there are $n!$ possible outcomes). The chosen outcome gives rise to
$$\frac{n(n-1)}{2}$$
inequalities. Then one must add $n-1$ inequalities expressing that the order was completely reversed by the plurality voting.
The volume of the corresponding polytope gives the probability of this event. By mutual
exclusion and symmetry, we must multiply the volume by $n!$ and then take the conditional probability under the hypothesis that such a linear order exists.

As a subpolytope of $\cU$, the polytope $\cB_{\St}$ is
defined by the inequalities in Table \ref{ineq3}.
\begin{table}[hbt]
{\small \tabcolsep=1.3pt
\begin{tabular}{rrrrrrrrrrrrrrrrrrrrrrrrr}
$\beta_1$:&$ 1$&$ 1$&$ 1$&$ 1$&$ 1$&$ 1$     &$-1$&$-1$&$-1$&$-1$&$-1$&$-1$      &$ 1$&$ 1$&$-1$&$-1$&$ 1$&$-1$        &$ 1$&$ 1$&$-1$&$-1$&$ 1$&$-1$\\
$\beta_2$:&$ 1$&$ 1$&$ 1$&$ 1$&$ 1$&$ 1$     &$ 1$&$ 1$&$-1$&$-1$&$ 1$&$-1$      &$-1$&$-1$&$-1$&$-1$&$-1$&$-1$        &$ 1$&$ 1$&$ 1$&$-1$&$-1$&$-1$\\
$\beta_3$:&$ 1$&$ 1$&$ 1$&$ 1$&$ 1$&$ 1$     &$ 1$&$ 1$&$ 1$&$-1$&$-1$&$-1$      &$ 1$&$ 1$&$ 1$&$-1$&$-1$&$-1$        &$-1$&$-1$&$-1$&$-1$&$-1$&$-1$\\
$\beta_4$:&$ 1$&$ 1$&$-1$&$-1$&$ 1$&$-1$     &$ 1$&$ 1$&$ 1$&$ 1$&$ 1$&$ 1$      &$-1$&$-1$&$-1$&$-1$&$-1$&$-1$        &$ 1$&$-1$&$ 1$&$ 1$&$-1$&$-1$\\
$\beta_5$:&$ 1$&$ 1$&$ 1$&$-1$&$-1$&$-1$     &$ 1$&$ 1$&$ 1$&$ 1$&$ 1$&$ 1$      &$ 1$&$-1$&$ 1$&$ 1$&$-1$&$-1$        &$-1$&$-1$&$-1$&$-1$&$-1$&$-1$\\
$\beta_6$:&\rule[-1.5ex]{0ex}{2ex}$ 1$&$-1$&$ 1$&$ 1$&$-1$&$-1$     &$ 1$&$-1$&$ 1$&$ 1$&$-1$&$-1$      &$1$&$1$&$1$&$1$&$1$&$1$        &$-1$&$-1$&$-1$&$-1$&$-1$&$-1$\\ \hline
$\beta_7$:&\rule{0ex}{3ex}$-1$&$-1$&$-1$&$-1$&$-1$&$-1$&$1$&$1$&$1$&$1$&$1$&$1$&$0$&$0$&$0$&$0$&$0$&$0$&$0$&$0$&$0$&$0$&$0$&$0$\\
$\beta_8$:&$0$&$0$&$0$&$0$&$0$&$0$&$-1$&$-1$&$-1$&$-1$&$-1$&$-1$&$1$&$1$&$1$&$1$&$1$&$1$&$0$&$0$&$0$&$0$&$0$&$0$\\
$\beta_9$:&$0$&$0$&$0$&$0$&$0$&$0$&$0$&$0$&$0$&$0$&$0$&$0$&$-1$&$-1$&$-1$&$-1$&$-1$&$-1$&$1$&$1$&$1$&$1$&$1$&$1$
\end{tabular}}
\vspace*{2ex} \caption{Inequalities for
$\cB_{\St}$}\label{ineq3}
\end{table}
Note that the first six inequalities describe the event that the result of the majority voting yields a linear order.
We have obtained
$$
\vol \cB_{\St}=\frac{1281727528386311499990911876166511}{25940255058441281524973174784000000000}.
$$

Finally, conditioned by the assumption that there exists a linear order on the result of the majority voting, we get the probability of the strict Borda paradox for large numbers of voters,
$$
B_{\St}= \frac{24\cdot\vol \cB_{\St}}{24\cdot\vol \cT}=
\frac{1281727528386311499990911876166511}{821261107784328041072841984000000000}\approx 0.00156.
$$

\begin{remark}\label{strict_Borda}
(a) The inequalities defining the strict Borda paradox have an obvious property, which we only state since it does not hold in the case of the strong Borda paradox discussed below: it does not matter if one considers the inequalities in Table \ref{ineq3} or the inequalities defined by the same linear forms multiplied by $-1$. In fact, the multiplication by $-1$ reverses the linear order on the candidates both for majority and plurality, and thus amounts to a renaming of the candidates.

(b) One may ask, as in \cite{GL2}, what happens if the \emph{negative plurality rule} is used instead of the plurality rule. The negative plurality rule requires the voters to cast a vote against their least preferred candidate. It is not difficult to see that inversion of the preference orders as discussed in subsection \ref{subsection:Cond_class} above, maps an event representing the strict Borda paradox for plurality to an event representing the strict Borda paradox for negative plurality. Therefore no new computation is necessary.

(c) With the notation used in \cite[Formula 19]{GL2}, we have
$$
B_{\St}=P_{\textup{StBR}}^{\textup{PR}} (4,k,\textup{IAC})=P_{\textup{StBR}}^{\textup{NPR}} (4,k,\textup{IAC}) \quad \text{ for } k\rightarrow \infty.
$$
We note that the probability of observing the strict Borda paradox in four candidates elections (under the Impartial Anonymous Culture hypothesis) is significantly smaller than the probability of observing the strict Borda paradox in three candidates elections, which was computed in \cite[Formula 19]{GL2} and is $1/90\approx 0.0111$.
\end{remark}

The strong Borda paradox is the voting situation in which there is an inversion between the winner or the loser from the \emph{majority voting} to \emph{plurality voting}.
It appears that de Borda was primarily concerned with the outcome that plurality voting
might elect the majority voting loser \cite{B}. In the following we say that the \emph{strong Borda paradox} occurs if the Condorcet loser is the winner of the plurality voting.
Let us assume that candidate $1$ is the Condorcet loser. Next, by assuming that $1$ wins the plurality we obtain $n-1$ new inequalities. Each candidate may fulfill these conditions, so by symmetry the result must be multiplied by $n$.
Finally, the conditional probability has to be computed.

As a subpolytope of $\cU$, the polytope $\cB_{\Sg}$ is
defined by the inequalities in Table \ref{ineq4}.
\begin{table}[hbt]
{\small \tabcolsep=1.3pt
\begin{tabular}{rrrrrrrrrrrrrrrrrrrrrrrrr}
$-\beta_1$:&$-1$&$-1$&$-1$&$-1$&$-1$&$-1$     &$ 1$&$ 1$&$ 1$&$ 1$&$ 1$&$ 1$      &$-1$&$-1$&$ 1$&$ 1$&$-1$&$ 1$        &$-1$&$-1$&$ 1$&$ 1$&$-1$&$ 1$\\
$-\beta_2$:&$-1$&$-1$&$-1$&$-1$&$-1$&$-1$     &$-1$&$-1$&$ 1$&$ 1$&$-1$&$ 1$      &$ 1$&$ 1$&$ 1$&$ 1$&$ 1$&$ 1$        &$-1$&$-1$&$-1$&$ 1$&$ 1$&$ 1$\\
$-\beta_3$:&\rule[-1.5ex]{0ex}{2ex}$-1$&$-1$&$-1$&$-1$&$-1$&$-1$     &$-1$&$-1$&$-1$&$ 1$&$ 1$&$ 1$      &$-1$&$-1$&$-1$&$ 1$&$ 1$&$ 1$        &$ 1$&$ 1$&$ 1$&$ 1$&$ 1$&$ 1$\\ \hline
$\beta_{10}$:&\rule{0ex}{3ex}$1$&$1$&$1$&$1$&$1$&$1$&$-1$&$-1$&$-1$&$-1$&$-1$&$-1$&$0$&$0$&$0$&$0$&$0$&$0$&$0$&$0$&$0$&$0$&$0$&$0$\\
$\beta_{11}$:&$1$&$1$&$1$&$1$&$1$&$1$&$ 0$&$ 0$&$ 0$&$ 0$&$ 0$&$ 0$&$-1$&$-1$&$-1$&$-1$&$-1$&$-1$&$0$&$0$&$0$&$0$&$0$&$0$\\
$\beta_{12}$:&$1$&$1$&$1$&$1$&$1$&$1$&$ 0$&$ 0$&$ 0$&$ 0$&$ 0$&$ 0$&$ 0$&$ 0$&$ 0$&$ 0$&$ 0$&$ 0$&$-1$&$-1$&$-1$&$-1$&$-1$&$-1$
\end{tabular}}
\vspace*{2ex} \caption{Inequalities for
$\cB_{\Sg}$}\label{ineq4}
\end{table}
One gets
$$
\vol \cB_{\Sg}=\frac{325451674835828550681491}{68475651442606080000000000}.
$$
Combined with the previous computations, we obtain the probability of the strong Borda paradox for large numbers of voters:
$$
B_{\Sg}= \frac{4\cdot\vol \cB_{\Sg}}{p_\CW}=
\frac{325451674835828550681491}{14352135440302080000000000}\approx 0.02268.
$$
(We have used that the probability that there exists a Condorcet loser equals the probability that there exists a Condorcet winner.)

The situation when the Condorcet winner is the loser of the plurality voting presents also interest, in which case we say that the \emph{reverse strong Borda paradox} appears.
This situation is easy to model: we simply have to multiply the linear forms of Table \ref{ineq4} by $-1$. In other words,
as a subpolytope of $\cU$, the polytope $\cB_{\SgRev}$ is
defined by the inequalities $\ \beta_1,\ \beta_2,\ \beta_4,-\beta_{10},-\beta_{11},-\beta_{12}$.
Its volume is given by
$$
\vol \cB_{\SgRev}=\frac{104898234852130241}{21035720123168587776}.
$$

Combined with the previous computations, we get the probability of the reverse strong Borda paradox for large numbers of voters
$$
B_{\SgRev}= \frac{4\cdot\vol \cB_{\SgRev}}{p_\CW}=
\frac{104898234852130241}{4408976007260798976}\approx 0.02379.
$$

Though the difference between the two forms of the strong Borda paradox may seem surprising, there are very simple arguments for the asymmetry, independent of any computations. See Remark \ref{order}(c).

\begin{remark}\label{strong_Borda}

(a) Again one may ask what happens if the negative plurality rule (NPR) is used for the strong Borda paradox (BP) and its reverse variant. In total we then have 4 variants. But the 2 new variants are isomorphic to those considered above via the inversion of preference orders:
\begin{center}
reverse BP with NPR $\cong$ BP,\qquad BP with NPR $\cong$ reverse BP.
\end{center}
In fact, the inversion of preference orders turns an event representing the strong Borda paradox for plurality  into an event for the reverse strong Borda paradox with negative plurality. Similarly it exchanges the events within the other pair.

(b) With the notation used in \cite[Formula 20]{GL2}, we have
\begin{align*}
B_{\Sg} &=P_{\textup{SgBR}}^{\textup{PR}} (4,k,\textup{IAC})\quad \text{and} \\
B_{\SgRev} &=P_{\textup{SgBR}}^\textup{NPR} (4,k,\textup{IAC}) \quad \text{ for } k\rightarrow \infty.
\end{align*}
We note that the probability of observing the strong Borda paradox in four candidates elections under the plurality rule (respectively the negative plurality rule) is smaller, but still of the same magnitude, with the probability of observing the strong Borda paradox in three candidates elections under the plurality rule (respectively the negative plurality rule), which was computed in \cite[Formula 20]{GL2} and is $4/135 \approx 0.0296$ (respectively $17/540 \approx 0.0315$).
\end{remark}

\begin{remark} After the initial submission of this paper to arXiv.org,
Lepelley informed us by e-mail that several volume computations
can also be done by a combination of the software packages LattE
\cite{LatInt} and lrs \cite{lrs}. He has obtained precisely the same
results as we. This is a good test for the correctness of all algorithms
involved.
\end{remark}

%%%%%%%%%%%%%%%%%%%%%%%%%%%%%%%%%%%%%%%%%%%%%%%%%%%%%%%%%%%%%%%%%%%%%%%%%%%%%%%%%%%%%%%%%%%%%%%%%%%%%%%%%%%%%%%%%%%%%%%%%%%%%%%%%%%%%%%%%%%%%%%%%%%

\section{Computations of Ehrhart series and quasipolynomials arising in
four candidates elections}\label{Ehrhart Computations}

While the probability of  a certain type of election result, for example
the Condorcet paradox, can be computed as the volume of a polytope
(or the sum of such volumes), one can use the polyhedral method also to
find the exact number of election results of the given type for a specific
number $k$ of voters. For example, if $\cC$ is the semiopen polytope
defined by the condition that candidate $1$ is the Condorcet winner,
then the number of election results for $k$ voters with Condorcet winner $1$ is
$$
E(\cC,k)=\#( k\cC\cap \ZZ_+^N).
$$
The function $E(\cC,k)$ is called the \emph{Ehrhart function} of
$\cC$.  The best approach to its computations uses the generating function
$$
E_\cP(t)=\sum_{k=0}^{\infty} E(\cC,k)t^k.
$$
This \emph{Ehrhart series} is the power series expansion of a rational function at
the origin. It is computed as a rational function, and in the following we will
always represent Ehrhart series in the form
numerator/denominator. The numerator is a polynomial with integer
coefficients. The denominator can always be written as a product of $d$
terms $1-t^g$, $d=\dim \cP+1$:
$$
E_\cP(t)=\frac{h_0+h_1t+\dots+h_st^s}{(1-t^{g_1})\cdots
(1-t^{g_d})},\qquad h_i\in\ZZ.
$$
Note that in general there exists no canonical representation in this
form; see Bruns, Ichim and S\"oger \cite[Section 4]{BIS} for a brief
discussion of this problem. If $\cP$ is closed, then $h_0=1$, and the
denominator can be chosen in such a way that all $h_i$ are nonnegative.
In the semiopen case such a representation may not exist. The theory of
Ehrhart series is developed in several books, for example in \cite{BG}.
For a treatment under the aspect of social choice we refer the reader to
\cite{LLS}.

For closed polytopes $\cP$, the Ehrhart function $E(\cP,k)$ itself is
given by a quasipolynomial $q_\cP$ for all $k\ge 0$. Roughly speaking,
this means that $q_\cP(k)$ is a polynomial whose coefficients depend
periodically on $k$. The period is a divisor of the least common
multiple of the exponents $g$ in the factors $1-t^g$ in the denominator.
In the Normaliz output the period is always exactly the least
common multiple. In the semiopen case one has $E(\cP,k)=q_\cP(k)$ only
for sufficiently large $k$. More precisely, if $r$ is the degree of
$E_\cP(t)$ as a rational function, then $E(\cP,r)\neq q_\cP(r)$, but
$E(\cP,k)= q_\cP(k)$ for all $k>r$.

For $n=4$ the Ehrhart series of $\cC$ has the numerator
\begin{align*}
&\phantom{+}6 t^{1} + 15 t^{2} + 481 t^{3} + 890 t^{4} + 12346 t^{5} +
17845 t^{6}\\
&+ 152891 t^{7} + 180850 t^{8} + 1113216 t^{9} + 1111974 t^{10} +5320122
t^{11}\\
&+ 4580485 t^{12} + 17837843 t^{13} + 13415068 t^{14} + 43770180 t^{15}
+ 28993857 t^{16}\\
&+ 80758791 t^{17} + 47336170 t^{18} + 113925878 t^{19} + 59177761
t^{20} + 123966919 t^{21}\\
&+ 56990048 t^{22} + 104272000 t^{23} + 42243510 t^{24} + 67509138
t^{25} + 23917200 t^{26}\\
&+ 33268048 t^{27} + 10182887 t^{28} + 12235441 t^{29} + 3176870 t^{30}
+3255226 t^{31}\\
&+ 697232 t^{32} + 596834 t^{33} + 100915 t^{34} + 69821 t^{35} + 8655
t^{36}\\
&+ 4581 t^{37} + 363 t^{38} + 133 t^{39} + 5 t^{40} + t^{41}.\\
\end{align*}
The Ehrhart series of $\overline{\cC}$ has the numerator
\begin{align*}
&\phantom{+}1+5t^{1}+133t^{2}+363t^{3}+4581t^{4}\\
&+8655t^{5}+69821t^{6}+100915t^{7}+596834t^{8}+697232t^{9}\\
&+3255226t^{10}+3176870t^{11}+12235441t^{12}+10182887t^{13}+33268048t^{14}\\
&+23917200t^{15}+67509138t^{16}+42243510t^{17}+104272000t^{18}+56990048t^{19}\\
&+123966919t^{20}+59177761t^{21}+113925878t^{22}+47336170t^{23}+80758791t^{24}\\
&+28993857t^{25}+43770180t^{26}+13415068t^{27}+17837843t^{28}+4580485t^{29}\\
&+5320122t^{30}+1111974t^{31}+1113216t^{32}+180850t^{33}+152891t^{34}\\
&+17845t^{35}+12346t^{36}+890t^{37}+481t^{38}+15t^{39}+6t^{40}.
\end{align*}
Both have the same denominator
$$
(1-t)(1-t^2)^{14}(1-t^4)^{9}.
$$
Numerator and denominator are coprime, but in general one cannot find a
coprime representation if one insists on a denominator that is a product
of terms $1-t^g$.

If we write the numerator of $\overline{\cC}$ as $\sum_{i=0}^{40}
a_it^i $, then the numerator of $\cC$ is $\sum_{i=0}^{40}
a_{40-i}t^{i+1}$: up to a shift in degree, they are palindromes of each
other. This rather unexpected relationship is not an accident, and is
explained by the following theorem.

\def\one{\mathbf 1}
\begin{theorem}\label{reciprocity}
Let $\lambda_1,\dots,\lambda_m$ be linear forms on $\RR^d$, and let
$\one\in \RR^d$ be the vector with the entry $1$ at all coordinates.
Suppose that $\lambda_i(\one)=0$ for all $i=1,\dots,m$. Set
$C=\{x\in\RR^d_+: \lambda_i(x)\ge0,\ i=1\dots,m\}$, and
$$
D=\{x\in\RR^d_+: \lambda_i(x)>0,\ i=1\dots,m\}.
$$
Define the he semiopen polytope $\cP$ by
$$
\cP=\{x\in D: \sum x_i=1\}.
$$
If $\dim \cP=d-1$ (the maximal dimension), then the following hold:
\begin{enumerate}
\item $J=I-\one$, where $J$ is the set of lattice points in $D$ and $I$
is the set of interior lattice points of $C$.
\item
$$
E_{\cP}(t)=(-1)^d\,t^{-d}E_{\overline\cP}(t^{-1}).
$$
\item Suppose that
$$
E_{\overline\cP}(t)=\frac{h_0+h_1t\dots+h_st^s}{(1-t^{g_1})\cdots(1-t^{g_d})}.
$$
Then the Ehrhart series of $\cP$ has the numerator polynomial
$$
h_st^w+\dots+h_0t^{w+s} ,\qquad w=\sum_{i=1}^d g_i-d-s,
$$
over the same denominator.
\end{enumerate}
\end{theorem}

\begin{proof}
The crucial observation is (1). Since $\dim \cP=d-1$, the interior of $C$ is
$$
\{x: x_i>0,\ i=1,\dots,d,\ \lambda_j(x)>0,\ j=1,\dots,m\}.
$$
For lattice points $x\in\ZZ^d$ these inequalities amount to $x_i\ge 1$
and $\lambda_j(x)\ge 1$. Thus $x$ belongs to the interior of $C$ if and
only if $x-\one$ satisfies the inequalities $x_i\ge 0$ for
$i=1,\dots,d$, and $\lambda_j(x)>0$ for $j=1,\dots,m$. This proves $J=I-\one$.

By Ehrhart reciprocity (for example, see \cite[Th. 6.51]{BG}) the
Ehrhart series of the interior of $\overline\cP$ is
$(-1)^dE_{\overline\cP}(t^{-1})$. In view of (1) we have to multiply
this series by $t^{-d}$ to obtain the Ehrhart series of $\cP$. This gives (2).

Part (3) is now an elementary transformation.
\end{proof}

\begin{remark}
(a) The condition $\lambda_i(\one)=0$ in Theorem \ref{reciprocity} is equivalent to the natural assumption that it only depends on the differences $v_i-v_j$ whether a voting result $(v_1,\dots,v_d)$ belongs to the event defined by the inequalities.

(b) We have formulated Theorem \ref{reciprocity} for the grading by
total degree. It can easily be generalized to other gradings.
\end{remark}

In the case of our polytopes we have $d=24$. So, for $\cC$ we obtain
$\sum_{i=1}^d g_i-d-s=65-24-40=1$, as observed. Theorem \ref{reciprocity}
is applicable to all polytopes in Section \ref{Volumes Computations} with
the exception of $\cF$. Nevertheless we have computed both Ehrhart series
in each case since the comparison is an excellent test of the Normaliz
algorithm. For $\cF$ the formula in Theorem \ref{reciprocity} does indeed
not hold.
\bigskip

The Ehrhart quasipolynomials of $\cC$ and $\overline{\cC}$
have period $4$. Moreover, they are equal for an odd number of voters
$k$ and have the same expression
$$
\#(\cC_k\cap\ZZ^N)=\#(\overline{\cC}_k\cap\ZZ^N)=
\frac{R_{1,3}(k)*(12 + k)*\prod_{i=1, i \text{ odd }}^{23} (i+k)
}{23!*131072}
$$
(for  $k\equiv 1,3 \text{ mod }4$), where
\begin{align*}
R_{1,3}(k)&=261812975764725 +308449567353120 k + 165347938576012 k^2 \\
&+ 50600971266720 k^3 + 9607752151310 k^4+ 1183838427360 k^5\\
&+ 96296973756 k^6+5130593760 k^7 + 172122725 k^8 + 3296640 k^9 + 27472
k^{10}.
\end{align*}

Let us reformulate this result in terms of probabilities. With the notation introduced at the beginning of Section \ref{Volumes Computations},
we have
$$
p_{\CW}^{(4)}(k)=4p_{\CW}^{(4,1)}(k),
$$
where
$$
p_{\CW}^{(4,1)}k(k)=\frac{\#(\cC_k\cap\ZZ^N)}{\#(\cU_k\cap \ZZ^N)}.
$$
Since
$$
\#(\cU_k\cap \ZZ^N)=\frac{\prod_{i=1}^{23} (i+k) }{23!}
$$
we get
$$
p_{\CW}^{(4)}(k)=\frac{R_{1,3}(k)*(12 + k)}{32768*\prod_{i=1, i \text{ even
}}^{23} (i+k)} \quad \text{ if }k\equiv 1,3 \text{ mod }4.
$$
This is exactly the same formula as the one computed first by
Gehrlein in \cite{G}. For the case of even $k$ we set
{\scriptsize
\begin{align*}
R_{0}(k)&=4981367114669230129152000 + 11069309139290261311979520 k+
11286725167650172468985856 k^2 \\
&+ 6970525765323041332002816 k^3 + 2896901556002851225731072 k^4 +
857336679021412589010944 k^5 \\
&+ 187293111169997407690752 k^6 + 30935327102400429176832 k^7 +
3923664152075008433664 k^8 \\
&+ 385511913998009006208 k^9 + 29422431828810359328 k^{10} +
1738486466127164288 k^{11} +\\
&+78715287099505056 k^{12} + 2678620940814672 k^{13} + 66260942646564
k^{14} \\
&+ 1124326347564 k^{15}+ 11698573833 k^{16} + 56262656 k^{17}
\end{align*}}
and
{\scriptsize
\begin{align*}
R_{2}(k)&=9794451243189989376000 + 921057250987916963020800 k +
1705900639387417842032640 k^2 \\
&+ 1489106767895973053595648 k^3 + 792353026020511342854144 k^4 +
284373446368099671547904 k^5 \\
&+ 72772788665361422238720 k^6 + 13747699097527641501696 k^7 +
1960073323091557035648 k^8 \\
&+ 213683286033339310848 k^9 + 17913763440866689440 k^10 +
1153396601212907264 k^11\\
&+ 56538334354261872 k^12 + 2071748534241792 k^13 + 54936786331200 k^14 \\
&+ 995421043392 k^15 + 11023421961 k^16 + 56262656 k^17,
\end{align*}}
then we get
$$
p_{\CW}^{(4)}(k)=
\begin{cases}
\frac{R_{0}(k)*k}{67108864*\prod_{i=0}^{5} (1+4*i+k)(2+4*i+k)(3+4*i+k)}
\quad \text{ if }k\equiv 0 \text{ mod }4;\\
  \\
\frac{R_{2}(k)*k}{67108864*\prod_{i=0}^{5} (4*i+k)(1+4*i+k)(3+4*i+k)}
\quad \quad \text{ if }k\equiv 2 \text{ mod }4.\\
\end{cases}
$$
To the best of our knowledge the above formula for an even number of voters
has not been computed before our computations with Normaliz.

\begin{remark} With the notation used in \cite[Formula 1.27 and
1.29]{GL}, we have
$$
p_{\CW}^{(4)}(k)=P_{\textup{PMRW}}^S(4,k,\textup{IAC}) \quad \text{ for all } k\in\ZZ_{+}.
$$
\end{remark}

For the other eight polytopes, since the numerators of the Ehrhart series are very long, we only list the denominators for a representation similar to the one given above for $\cC$ and $\overline{\cC}$ (with non-coprime numerators):
\begin{align*}
\cQ,\overline{\cQ}:&\ (1-t)(1-t^2)^2(1-t^4)^5(1-t^{12})^{16},\\
\cE,\overline{\cE},
\cF,\overline{\cF},
\cB_{\Sg},\overline{\cB_{\Sg}}:&\ (1-t)(1-t^2)^2(1-t^4)^5(1-t^{12})^4(1-t^{24})(1-t^{120})^{11},\\
\cT,\overline{\cT}:&\ (1-t)(1-t^2)^{14}(1-t^4)^5(1-t^{12})^{3}(1-t^{24}),\\
\cK,\overline{\cK}:&\ (1-t)(1-t^2)^{14}(1-t^4)^5(1-t^{12})^{4},\\
\cB_{\St},\overline{\cB_{\St}}:&\ (1-t)(1-t^2)^2(1-t^{4})^5(1-t^{12})^4(1-t^{24})(1-t^{120})^4\\
                                         &\ (1-t^{840})^2(1-t^{2520})^2(1-t^{27720})^{2}(1-t^{55440}),\\
\cB_{\SgRev},\overline{\cB_{\SgRev}}:&\ (1-t)(1-t^2)^{2}(1-t^4)^6(1-t^{12})^{3}(1-t^{24})^{12}.
\end{align*}
The denominators have $24$ factors. All computed data is available and will be provided by the authors on request.

The reciprocity between $E_\cP(t)$ and $E_{\overline \cP}$ in Theorem \ref{reciprocity} can be recast into a relation between the Ehrhart quasipolynomials. In terms of quasipolynomials, Ehrhart reciprocity says $q_{\relint \overline\cP}(k)=(-1)^{d-1}q_{\overline\cP}(-k)$ for all $k\ge 1$ (for example, see \cite[Th. 6.51]{BG}), and in view of Theorem \ref{reciprocity} this implies
$$
q_\cP(\ell)=(-1)^{d-1}(-\ell-d).
$$
It follows that under the conditions of Theorem \ref{reciprocity} one
has $E(\cP,k)=q_\cP(k)$ for all $k>-d$, and therefore for all $k>0$.

\begin{remark}\label{order}
(a) In Table \ref{numerator_data} we summarize the essential data of the numerators of the Ehrhart series of the polytopes (with the exception of $\cF$), according to the notations introduced in Theorem \ref{reciprocity}.
The last column represents the period of the associated Ehrhart quasipolynomials.

\begin{table}[hbt]
\centering
\begin{tabular}{|r|r|r|r|r|r|}
\hline
\rule[-0.1ex]{0ex}{2.5ex}
Polytope & $d$ & $s$ & $\sum g_i$ & $w$ & period \\ \hline
\strut $\cC$          & 24 &     40 &     65 & 1 & 4 \\ \hline
\strut $\cQ$          & 24 &    190 &    217 & 3 & 12\\ \hline
\strut $\cE$          & 24 &   1392 &   1417 & 1 & 120\\ \hline
\strut $\cT$          & 24 &     84 &    109 & 1 & 24\\ \hline
\strut $\cK$          & 24 &     70 &     97 & 3 & 12\\ \hline
\strut $\cB_{\St}$    & 24 & 118144 & 118177 & 9 & 55440\\ \hline
\strut $\cB_{\Sg}$    & 24 &   1388 &   1417 & 5 & 120\\ \hline
\strut $\cB_{\SgRev}$ & 24 &    326 &    353 & 3 & 24\\ \hline
\end{tabular}
\vspace*{2ex} \caption{Data for the numerators of the Ehrhart series and quasipolynomials}\label{numerator_data}
\end{table}

(b) The numerator of $\overline{\cF}$ is a polynomial of degree $1386$ whereas the numerator of $\cF$ has degree $1389$.
They are not related by Theorem \ref{reciprocity}, so the Ehrhart series must be computed separately by Normaliz.
The Ehrhart quasipolynomials of $\cF$ and
$\overline{\cF}$  have period $120$.

(c) The number $w$ in Table \ref{numerator_data} is the smallest number of voters that can realize the respective election outcome since it is the lowest exponent of $t$ that appears in the Ehrhart series (when written as a power series).  Note the different values of $\cB_{\Sg}$ and $\cB_{\SgRev}$ for $w$.

The values of $w$ can be checked by elementary arguments, independently from the Ehrhart series computation.  We explain it for the strong Borda paradox. Let candidate 1 be the plurality winner. Then he has got more first places, say $m$, than any of the three other candidates. On the other hand, he is the Condorcet loser, and therefore there must be behind any of the other candidates in at least $m+1$ preference orders. Another candidate, for example 2, gets first place in at least one of these $m+1$ preference orders. Then $m>1$, so $m\ge 2$. For the total number of voters $k$ it follows that $k \ge 2m+1\ge 5$. But the strong Borda paradox can be realized by $5$ voters: two of them place candidate 1 first, and the other three place the different other 3 candidates first, but candidate 1 last. Note that the same argument works for any number of candidates, with the exception of the three candidates situation where the minimal number of voters is 7. Similarly one can see that, for four candidates elections, the reverse strong Borda paradox requires only 3 voters and the strict Borda paradox needs 9 of them.

\end{remark}

\section{The exploitation of symmetry}\label{Exploit}

The elegant approach of Schürmann in \cite{Sch} for the
computation of the volumes of $\cC$ and variants of $\cQ$ and
$\cE$ uses the high degree of symmetries of these
polytopes. Suppose that a polytope $P\subset \RR^d$ is defined by the sign inequalities $x_i\ge 0$, $i=1,\dots,d$, and further inequalities $\lambda_i(x)\ge 0$ (or $>0$).
If certain variables $x_{i_1},\dots,x_{i_u}$ occur
in all of the linear forms $\lambda_i$ with the same coefficient (that may depend on $i$), then any permutation of them acts as a symmetry on
the corresponding polytope, and the variables
$x_{i_1},\dots,x_{i_u}$ can be replaced by their sum
$y_j=x_{i_1}+\dots+x_{i_u}$ in the $\lambda_i$. (The polytopes may have further
symmetries.) The substitution can be used for a projection into
a space of much lower dimension, mapping the polytope $P$
to a polytope $Q$ (this requires that the grading
affine hyperplane $A_1$ is mapped onto an affine hyperplane by
the projection). Instead of counting the lattice points in $kP$
one counts the lattice points in $kQ$ weighted with their
number of preimage lattice points in $kP$. This amounts to the
consideration of a \emph{weighted Ehrhart function}
$$
k\mapsto \sum_{y\in kQ\cap\ZZ^d} f(y).
$$
The polynomial $f$ is easily determined by elementary combinatorics: if $y_j$, $j=1,\dots,m$,  is the sum of  $u_j$ variables $x_i$, then
$$
f(y)=\binom{u_1+y_1-1}{u_1-1}\cdots \binom{u_m+y_m-1}{u_m-1}.
$$

The theory of weighted Ehrhart functions has recently been
developed in several papers; see Baldoni et al. \ \cite{BB1},
\cite{BB2}, \cite{BS}. In \cite{Sch}, only the leading form of the polynomial $f$ is
used. Integration of the leading form with respect to Lesbesgue measure yields
the volume.

In the case of the Condorcet polytope $\cC$ for four candidates the symmetrization yields a polytope of dimension $7$: there are two groups of $6$ variables each that can be replaced by their sums, and 6 groups of two variables each. We leave it to the reader to spot them. Fortunately Normaliz finds them automatically. The polynomial $f$, also computed automatically, is
$$
f(y)=\binom{y_1+5}{5}(y_2+1)(y_3+1)(y_4+1)(y_5+1)(y_6+1)(y_7+1)\binom{y_8+5}{5}
$$
in this case. Among our polytopes, only  $\cT$ and $\cB_{\St}$ do not allow any symmetrization.

Version 3.2.1 of Normaliz calls its offspring NmzIntegrate behind the scenes for the symmetrized computation of volumes and weighted Ehrhart series. In version 3.3.0, NmzIntgerate will be included in Normaliz itself and no longer exist as a separate program.
The algorithmic approach of NmzIntegrate is developed in \cite{BS}. For polynomial arithmetic Normaliz uses CoCoALib by Abbott, Bigatti and Lagorio \cite{CoCoA}.

\section{Computational report}\label{CompRep}
In this section we want to document the use of Normaliz 3.2.1 and computations performed with it during the preparation of this work.

\subsection{Use of Normaliz}
Normaliz is distributed as open source under the GPL. In
addition to the source code, the distribution contains
executables for the major platforms Linux, Mac and Windows.
We include some details on the use of Normaliz in order to show that the input files have a transparent structure and that the syntax of the execution command is likewise simple.

The polytope $\cF$ has the following input file:
{\small
\begin{verbatim}
amb_space 24
excluded_faces 5
1 1 1 1 1 1   -1 -1 -1 -1 -1 -1    1  1 -1 -1  1 -1    1  1 -1 -1  1 -1
1 1 1 1 1 1    1  1 -1 -1  1 -1   -1 -1 -1 -1 -1 -1    1  1  1 -1 -1 -1
1 1 1 1 1 1    1  1  1 -1 -1 -1    1  1  1 -1 -1 -1   -1 -1 -1 -1 -1 -1

1 1 1 1 1 1    0  0   0 0  0  0   -1 -1 -1 -1 -1 -1    0  0  0  0  0  0
1 1 1 1 1 1    0  0   0 0  0  0    0  0  0  0  0  0   -1 -1 -1 -1 -1 -1

inequalities 1
-1 -1 -1 -1 -1 -1    1  1 1 1 1 1     0 0 0 00 0     0 0 0 0 0 0
nonnegative
total_degree
\end{verbatim}
}

The first line \verb|amb_space 24| sets the ambient space to $\RR^{24}$. The $5$ \verb|excluded_faces| represent the strict inequalities $\lambda_i>0$, $i=1,2,3,5,6$ of Table \ref{ineqadded}, whereas the non-strict inequality $\overline{-\lambda_4}(v)\ge 0$ is given by $1$ \verb|inequalities|.
The keyword \verb|nonnegative| indicates that all $24$ coordinates are to be taken nonnegative, whereas \verb|total_degree| defines the grading in which each coordinate has weight $1$.
Let us suppose the file is called \verb|CEP.in|.

For Normaliz 3.2.1 and newer versions the  simplest command for the computation is
\begin{verbatim}
normaliz CEP
\end{verbatim}
Depending on the installation, it may be necessary to prefix \verb|normaliz| or \verb|CEP| by a path. Often one adds the option \verb|-c| to get terminal output showing the progress of the computation. If no further option is added, Normaliz will compute the Hilbert series and the Hilbert basis. The computation of the latter is very fast in all cases of this paper, but one can suppress it by the option \verb|-q|. If one is only interested in the volume, \verb|-v| is the right choice. The number of parallel threads can be limited via the option \verb|-x=<N>| where \verb|<N>| stands for the number of threads. The computation results are contained in \verb|CEP.out|.

As will be apparent from the terminal output (obtained with \verb|-c|), Normaliz successfully tries symmetrization, and employs its companion NmzIntegrate as mentioned out above.
One should note that Normaliz does automatic symmetrization only if the cone $D$ defined by image $\cQ$ has dimension $\le 2\dim(C)/3$ where $C$ is the cone over the polytope $\cP$ to be computed. The bound has been introduced since  one cannot expect a saving in computation time if the dimension does not drop enough. However, the user can force Normaliz to use symmetrization.

\begin{remark}
(a) Normaliz has an input type \verb|strict_inequalities|. While it seems a natural choice and will yield the same results as \verb|excluded_faces|, its use for the computations of this paper is not advisable since the algorithmic approach does not (yet) allow symmetrization, and even for cases without symmetrization it is usually significantly slower than \verb|excluded_faces|.

(b) A graphical interface called jNormaliz (Almendra and Ichim \cite{AI}) is also available in the Normaliz package. For its use Java must be installed on the system.
\end{remark}

\subsection{Overview of the examples}

The columns of Table \ref{data} contain the values of
characteristic numerical data of the examples studied, namely:
$\#$vertices is the number of the vertices of the polytope, and $\#$supp the number of its support
hyperplanes. $\#$Hilb is the size of the Hilbert basis of the Ehrhart cone over the polytope (see \cite{BI} for more details). These data are invariants of the polytope.

\begin{table}[hbt]
\centering
\begin{tabular}{|r|r|r|r|r|r|}\hline
\rule[-0.1ex]{0ex}{2.5ex}Polytope & $\#$vertices & $\#$supp& $\#$Hilb & $\#$ triangulation & $\#$ Stanley dec\\ \hline
\strut $\cC$          &   234 & 27 &    242 &       1,344,671 & 1,816,323\\ \hline
\strut $\cQ$          & 2418  & 28 & 12,948 & 343,466,918,256 & 2,217,999,266,634\\ \hline
\strut $\cE$          & 4644  & 30 & 31,308 & 464,754,352,804 & 1,661,651,089,155\\ \hline
\strut $\cF$          & 4572  & 30 & 26,325 & 1,009,992,718,827 & 3,400,149,589,030\\ \hline
\strut $\cT$          & 491   & 30 &  546   & 2,852,958 & 5,635,927\\ \hline
\strut $\cK$          & 262   & 28 & 362    & 1,346,894 & 2,694,560\\ \hline
\strut $\cB_{\St}$    & 6363  & 33 & 21,137 & 30,399,162,846 & 75,933,588,203\\ \hline
\strut $\cB_{\Sg}$    & 3216  & 30 & 24,816 & 149,924,230,551 & 858,660,657,413\\ \hline
\strut $\cB_{\SgRev}$ & 3432  & 30 & 9,548  & 366,864,865,269 & 1,141,025,866,136\\ \hline
\end{tabular}
\vspace*{2ex} \caption{Numerical data of test
examples}\label{data}
\end{table}

The last two columns list the number of simplicial cones in the
triangulation and the number of components of the Stanley
decomposition (see \cite{BIS} for details on these numbers). These data are not invariants of the polytope.
The information is included to show the complexity of the computations if symmetrization is not used.
Normaliz can do all computations without symmetrization, but then some of them will take days,
even those with a high degree of symmetry.
The size of the lexicographic triangulation depends on the
order in which the extreme rays are processed. The polytopes in
the table above are defined by their support hyperplanes, and
therefore Normaliz first computes the extreme rays from them.
The order used in the computations mentioned in Table
\ref{data} is not necessarily identical with the order
produced by previous versions of Normaliz. Moreover,
bottom decomposition, see \cite{BSS}, is used automatically if the ratio
of the largest degree among the generators and the smallest is $\ge 10$.
This further influences the data contained in the last two columns.

Of all our polytopes, only  $\cT$ and $\cB_{\St}$ cannot be symmetrized. The combinatorial data of the symmetrized polytopes are contained in Table \ref{symm_data}.

\begin{table}[hbt]
\centering
\begin{tabular}{|r|r|r|r|r|r|}\hline
\rule[-0.1ex]{0ex}{2.5ex}Polytope & dim &$\#$vertices & $\#$supp& $\#$
triangulation & $\#$ Stanley dec\\ \hline
\strut $\cC$          &  8 &  16 & 11 & 17 &    33\\ \hline
\strut $\cQ$          &  6 &  12 &  8 & 14 &    14 \\ \hline
\strut $\cE$          & 13 & 170 & 18 & 18208 & 19999\\ \hline
\strut $\cF$          & 13 & 163 & 18 & 23738 & 41963\\ \hline
\strut $\cK$          & 14 &  63 & 18 & 1035 &  2070\\ \hline
\strut $\cB_{\Sg}$    & 13 & 100 & 18 & 3696 &  6025\\ \hline
\strut $\cB_{\SgRev}$ & 13 & 115 & 19 & 10342 & 26024\\ \hline
\end{tabular}
\vspace*{2ex} \caption{Numerical data of symmetrized
polytopes}\label{symm_data}
\end{table}

We remark that, for Hilbert basis computations, the dual algorithm of Normaliz (see \cite{BI}) is much faster than the primal
algorithm for the examples of this paper, and all computations run in a few seconds.
This is by no means always the case (see \cite{BIS}).

\subsection{Hardware characteristics}

Almost all computations were run on a compute server with operating system CentOS 7.3, 4 Intel
Xeon  E5-2660 at 2.20GHz (a total of 32 cores) and 192 GB of
RAM. With the exception of $\cB_{\St}$, all computations were also done on a standard laptop with operating system Ubuntu 16.04, an Intel i5-4200M CPU at 2.5 Ghz and 12 GB RAM. In parallelized computations we
have limited the number of threads used to $30$.  In
Tables \ref{vol_data} and \ref{ES_data}, \ttt{4x} and \ttt{30x} indicates parallelization with $4$ and $30$ threads, respectively. As the large
examples below show, the parallelization scales efficiently; see also Table \ref{eff_parall}. The version that we have used exchanges data via files.
The laptop has an SSD , but the server has only hard disks, and it is not a local hard disk of the machine,
but the files must go through the NFS, the network file system.

Normaliz and NmzIntegrate need relatively little memory. All computations
mentioned in this paper run stably with $< 0.5$ GB of RAM for each thread used.

\subsection{Volumes}

Table \ref{vol_data} contains the computation times for the
volumes of the studied polytopes.
The input for all these examples is given in the form of inequalities. When one runs Normaliz
on such examples, it first computes the extreme rays of the cone
and uses them as generators. This small  extra time is also included in the reported times below. (It is also apparent and not surprising that small examples profit from a small number of parallel threads.)

\begin{table}[hbt]
\centering
\begin{tabular}{|r|c|r|r|}
\hline
\rule[-0.1ex]{0ex}{2.5ex}
Polytope & Symmetrize & Laptop \ttt{4x} & Server \ttt{30x} \\ \hline
\strut $\cC$          &   Yes & 0.100 s    & 0.591 s\\ \hline
\strut $\cQ$          &   Yes & 0.33 s     & 0.76 s\\ \hline
\strut $\cE$          &   Yes & 1:11:39 h  & 8:41 m\\ \hline
\strut $\cF$          &   Yes & 1:48:57 h  & 15:06 m\\ \hline
\strut $\cT$          &    No & 7.200 s    & 10.455 s\\ \hline
\strut $\cK$          &   Yes & 0.660 s    & 1.940 s\\ \hline
\strut $\cB_{\St}$    &    No & --         & 3:57:26 h\\ \hline
\strut $\cB_{\Sg}$    &   Yes & 14:51 m    & 1:39 m\\ \hline
\strut $\cB_{\SgRev}$ &   Yes & 44:54 m    & 4:17 m\\ \hline
\end{tabular}
\vspace*{2ex} \caption{Computation times for volumes}\label{vol_data}
\end{table}

In order to measure the parallelization we have run the volume
computation of $\cE$  with varying number of threads. Table
\ref{eff_parall} shows that NmzIntegrate is very efficiently parallelized.
\begin{table}[hbt]
\centering
\rule[-0.1ex]{0ex}{2.5ex}
\begin{tabular}{|c||r|r|r|r|r|}
\hline
\strut\# threads & 1 & 5 & 10 & 20 & 30\\
\hline
\strut real time s & 10367 & 2380 & 1245 & 656 & 521 \\
\hline
\strut efficiency \% & 100 & 88 & 83 & 79 & 66 \\
\hline
\end{tabular}
\vspace*{2ex} \caption{Efficiency of parallelization in volume computations}\label{eff_parall}
\end{table}

\begin{remark}\label{remark_ASch} The volume of the polytope $\cC$ was first computed by Gehrlein \cite{G}. The volumes of variants of the polytopes $\cQ$ and $\cE$ had been
computed by Schürmann \cite{Sch} with LattE integrale \cite{LatInt}. This information
was very useful for checking the correctness of Normaliz.
\end{remark}

\subsection{Ehrhart series and quasipolynomials} The experimental times obtained
for computation of Ehrhart series and quasipolynomials are contained in Table \ref{ES_data}.
As above, the presented times include the time used by Normaliz for computing the extreme rays of the cone.
Moreover, the Ehrhart quasipolynomials are computed from the Ehrhart series (see \cite{BI}).
This requires for some examples like  $\cB_{\St}$ a significant extra time, which has likewise been included. We have also measured the parallelization for the Ehrhart series computation of $\cB_{\St}$; see Table \ref{Ehrhart parallel}. Somewhat surprisingly, the efficiency is $>100\%$ for certain numbers of threads, an effect that can only be explained by the memory management of the system.

\begin{table}[hbt]
\centering
\begin{tabular}{|r|r|r|r|r|r|}
\hline
\rule[-0.1ex]{0ex}{2.5ex}
Polytope & Symmetrize & \multicolumn{2}{ |c| }{Laptop \ttt{4x}} & \multicolumn{2}{ |c| }{Server \ttt{30x}} \\
\cline{3-6}
         &  & closed & semi-open  & closed & semi-open \\ \hline
\strut $\cC$          &   Yes & 1.730 s    & 1.940 s    & 1.925 s    & 2.077 s    \\ \hline
\strut $\cQ$          &   Yes & 4.400 s    & 7.64 s     & 7.010 s    & 8.440 s    \\ \hline
\strut $\cE$          &   Yes & 4:50:55 h  & 4:45:24 h  & 28:36 m    & 41:01 m    \\ \hline
\strut $\cF$          &   Yes & 12:02:42 h & 12:47:03 h & 1:45:15 h  & 1:39:19 h \\ \hline
\strut $\cT$          &    No & 16.230 s   & 28.260 s   & 24.050 s   & 34.136 s   \\ \hline
\strut $\cK$          &   Yes & 16.770 s   & 25.810 s   & 3.156 s    & 7.967 s    \\ \hline
\strut $\cB_{\St}$    &    No & --         & --         & 10:08:50 h & 37:03:26 h \\ \hline
\strut $\cB_{\Sg}$    &   Yes & 1:34:23 h  & 1:36:16 h  & 9:01 m     & 14:13 m    \\ \hline
\strut $\cB_{\SgRev}$ &   Yes & 5:56:18 h  & 5:53:38 h  & 45:13 m    & 47:22 m    \\ \hline
\end{tabular}
\vspace*{2ex} \caption{Computation times for Ehrhart series and quasipolynomials}\label{ES_data}
\end{table}

\begin{table}[hbt]
\centering
\rule[-0.1ex]{0ex}{2.5ex}
\begin{tabular}{|c||r|r|r|r|r|}
\hline
\strut\# threads & 1 & 5 & 10 & 20 & 30\\
\hline
\strut real time s & 15959 & 398 & 1230 & 635 & 553 \\
\hline
\strut efficiency \% & 100 & 94 & 130 & 126 & 96 \\
\hline
\end{tabular}
\vspace*{2ex} \caption{Efficiency of parallelization in Ehrhart series computations}\label{Ehrhart parallel}
\end{table}

\section{Acknowledgement}

The authors like to thank  Achill Sch\"{u}rmann for several test examples that were used during the
development of Normaliz and NmzIntegrate.

\end{document}